\documentclass[12pt]{amsart}
\usepackage{amssymb,bbold}
\usepackage[all]{xy}

\theoremstyle{plain}
\newtheorem{theorem}{Theorem}
\newtheorem{corollary}[theorem]{Corollary}
\newtheorem{lemma}[theorem]{Lemma}

\theoremstyle{definition}

\newtheorem{remark}[theorem]{Remark}

\newcommand{\abs}[1]{\lvert#1\rvert}
\newcommand{\norm}[1]{\lVert#1\rVert}
\newcommand{\bigabs}[1]{\bigl\lvert#1\bigr\rvert}
\newcommand{\bignorm}[1]{\bigl\lVert#1\bigr\rVert}

\renewcommand{\le}{\leqslant}
\renewcommand{\ge}{\geqslant}
\renewcommand{\mid}{\::\:}

\newcommand{\term}[1]{{\textit{\textbf{#1}}}}

\def\api{{\abs{\!\pi\!}}}
\def\ftp{\otimes_{\!\scriptscriptstyle\api}\!}
\def\one{\mathbb 1}

\def\Io{I_{\rm o}}
\def\Ioc{I_{\rm oc}}

\DeclareMathOperator{\sign}{sign}

\begin{document}
\baselineskip 18pt

\title[Fremlin tensor products of concavifications]
      {Fremlin tensor products of concavifications of Banach lattices}

\author[V.G.~Troitsky]{Vladimir G. Troitsky}
\author[O.~Zabeti]{Omid Zabeti}

\address[V.G.~Troitsky]
  {Department of Mathematical
  and Statistical Sciences, University of Alberta, Edmonton,
  AB, T6G\,2G1. Canada}
\email{troitsky@ualberta.ca}

\address[O.~Zabeti]
  {Department of Pure Mathematics, Ferdowsi University of Mashhad,
   P.O. Box 1159, Mashhad 91775, Iran}
\email{ozabeti@yahoo.ca}

\thanks{The first author was supported by NSERC}

\begin{abstract}
  Suppose that $E$ is a uniformly complete vector lattice and
  $p_1,\dots,p_n$ are positive reals. We prove that the diagonal of
  the Fremlin projective tensor product of $E_{(p_1)},\dots,E_{(p_n)}$
  can be identified with $E_{(p)}$ where $p=p_1+\dots+p_n$ and
  $E_{(p)}$ stands for the $p$-concavification of $E$. We also provide
  a variant of this result for Banach lattices. This extends the main
  result of~\cite{BBPTT}.
\end{abstract}

\date{\today}

\maketitle

\section{Introduction and motivation}

We start with some motivation. Let $E$ be a vector or a Banach lattice
of functions on some set $\Omega$, and consider a tensor product
$E\tilde\otimes E$ of $E$ with itself. It is often possible to view
$E\tilde\otimes E$ as a space of functions on the square
$\Omega\times\Omega$ with $(f\otimes g)(s,t)=f(s)g(t)$, where $f,g\in
E$ and $s,t\in\Omega$. In particular, restricting this function to the
diagonal $s=t$ gives just the product $fg$. Thus, the space of the
restrictions of the elements of $E\tilde\otimes E$ to the diagonal can
be identified with the space $\{fg\mid f,g\in E\}$, which is sometimes
called the square of $E$, (see, e.g.,~\cite{Buskes:01}). The concept
of the square can be extended to uniformly complete vector lattices as
the 2-concavification of $E$. Hence, one can expect that, for a
uniformly complete vector lattice, the diagonal of an appropriate
tensor product of $E$ with itself can be identified with the
2-concavification of $E$. This was stated and proved formally
in~\cite{BBPTT} for Fremlin projective tensor product of Banach
lattices. It was shown there that the diagonal of the tensor product
is lattice isometric to the 2-concavification of $E$; the diagonal was
defined as the quotient of the product over the ideal generated by all
elementary tensors $x\otimes y$ with $x\perp y$.

In the present paper, we extend this result. Let us again
provide some motivation. In the case when $E$ is a space of
functions on $\Omega$, one can think of its $p$-concavification as
$E_{(p)}=\{f^p\mid f\in E\}$ for $p>0$. In this case, for positive
real numbers $p_1,\dots,p_n$, the elementary tensors in
\begin{math}
  E_{(p_1)}\otimes\dots\otimes E_{(p_n)}
\end{math}
can be thought of as functions on $\Omega^n$ of the form
\begin{math}
  f_1^{p_1}\otimes\dots\otimes f_n^{p_n}(s_1,\dots,s_n)
  =f_1^{p_1}(s_1)\cdots f_n^{p_n}(s_n),
\end{math}
where $f_1,\dots,f_n\in E$ and $s_1,\dots,s_n\in\Omega$. The
restriction of this function to the diagonal is the product
$f_1^{p_1}\cdots f_n^{p_n}$, which is an element of $E_{(p)}$ where
$p=p_1+\dots+p_n$. That is, the diagonal of the tensor product of
$E_{(p_1)},\dots,E_{(p_n)}$ can be identified with $E_{(p)}$. In this
paper, we formally state and prove this fact for the case when $E$ is
a uniformly complete vector lattice in Section~\ref{VL} and when $E$
is a Banach lattice in Section~\ref{BL}.  In particular, this
extends the result of~\cite{BBPTT} to vector lattices and to the
product of an arbitrary number of copies of $E$ (a variant of the
latter statement was also independently obtained in~\cite{BB}).

\section{Products of vector lattices}
\label{VL}

Throughout this section, $E$ will stand for a uniformly
complete vector lattice. We need uniform completeness so that we can
use positive homogeneous function calculus in $E$, see, e.g.,
Theorem~5 in~\cite{Buskes:01}. It is easy to see that every uniformly
complete vector lattice is Archimedean.

Following~\cite[p.~53]{Lindenstrauss:79}, by $t^p$, where $t\in\mathbb
R$ and $p\in\mathbb R_+$, we mean $\abs{t}\cdot\sign{t}$; see also the
discussion in Section~1.2 of~\cite{BBPTT}. In particular, if
$p_1,\dots,p_n$ are positive reals and $p=p_1+\dots+p_n$ then
\begin{math}
  \abs{t}^{p_1}\abs{t}^{p_2}\dots\abs{t}^{p_n}=\abs{t}^p
\end{math}
while
\begin{math}
  t^{p_1}\abs{t}^{p_2}\dots\abs{t}^{p_n}=t^p
\end{math}
for every $t\in R$. It follows that
\begin{math}
  \abs{x}^{p_1}\abs{x}^{p_2}\dots\abs{x}^{p_n}=\abs{x}^p
\end{math}
and
\begin{math}
  x^{p_1}\abs{x}^{p_2}\dots\abs{x}^{p_n}=x^p
\end{math}
for every $x\in E$. In particular,
\begin{math}
  (x\abs{x}\cdots\abs{x})^{\frac{1}{n}}=x
\end{math}
for every $n\in\mathbb N$.

Suppose that $p$ is a positive real number. Using function
calculus, we can introduce new vector operations on $E$ via $x\oplus
y=(x^p+y^p)^{\frac{1}{p}}$ and $\alpha\odot x=\alpha^{\frac{1}{p}}x$,
where $x,y\in E$ and $\alpha\in\mathbb R$. Together with these new
operations and the original order and lattice structures, $E$ becomes
a vector lattice. This new vector lattice is denoted $E_{(p)}$ and
called the $p$-\term{concavification} of $E$. It is easy to see that
$E_{(p)}$ is still Archimedean.

We start by extending Theorem~1 and Corollary~2
in~\cite{Buskes:00}. Recall that an $n$-linear map $\varphi$ from
$E^n$ to a vector lattice $F$ is said to be \term{positive} if
$\varphi(x_1,\dots,x_n)\ge 0$ whenever $x_1,\dots,x_n\ge 0$ and
\term{orthosymmetric} if $\varphi(x_1,\dots,x_n)=0$ whenever
$\abs{x_1}\wedge\dots\wedge\abs{x_n}=0$; $\varphi$ is said to be a
\term{lattice $n$-morphism} if
\begin{math}
  \bigabs{\varphi(x_1,\dots,x_n)}=
  \varphi\bigl(\abs{x_1},\dots,\abs{x_n}\bigr)
\end{math}
for any $x_1,\dots,x_n\in E$.

\begin{theorem}\label{CK-ones}
  Suppose that $\varphi\colon C(K)^n\to F$, where $K$ is a compact
  Hausdorff space, $F$ is a vector lattice, $n\in\mathbb N$, and
  $\varphi$ is an orthosymmetric positive $n$-linear map. Then
  $\varphi(x_1,\dots,x_n)=\varphi(x_1\cdots x_n,\one,\dots,\one)$ for
  any $x_1,\dots,x_n\in C(K)$.
\end{theorem}

\begin{proof}
  The proof is by induction. The case $n=1$ is trivial. The case $n=2$
  follows from Theorem~1 in~\cite{Buskes:00}. Suppose that $n>2$ and
  the statement is true for $n-1$.  Suppose that $\varphi\colon
  C(K)^n\to F$ is orthosymmetric positive and $n$-linear. Fix $0\le
  z\in C(K)$ and define $\varphi_z\colon C(K)^{n-1}\to F$ via
  \begin{math}
    \varphi_z(x_1,\dots,x_{n-1})=\varphi(x_1,\dots,x_n,z).
  \end{math}
  Clearly, $\varphi_z$ is orthosymmetric, positive, and
  $(n-1)$-linear. By the induction hypothesis,
  \begin{math}
    \varphi_z(x_1,\dots,x_{n-1})=\varphi_z(x_1\cdots x_{n-1},\one,\dots,\one)
  \end{math}
  for all $x_1,\dots,x_{n-1}$ in $C(K)$.
  It follows that
  \begin{equation}
  \label{mid-ones}
    \varphi(x_1,\dots,x_n)=\varphi(x_1\cdots x_{n-1},\one,\dots,\one,x_n)
  \end{equation}
  for all $x_1,\dots,x_{n-1}$ and all $x_n>0$. By linearity,
  \eqref{mid-ones} remains true for all $x_1,\dots,x_n\in C(K)$ as
  $x_n=x_n^+-x_n^-$. Similarly,
  \begin{math}
    \varphi(x_1,\dots,x_n)=\varphi(x_1x_3\cdots x_n,x_2,\one,\dots,\one)
  \end{math}
  for all $x_1,\dots,x_n$ in $E$.
  Applying \eqref{mid-ones} to the latter expression, we get
  \begin{math}
    \varphi(x_1,\dots,x_n)=\varphi(x_1\cdots x_n,\one,\dots,\one).
  \end{math}
  This completes the induction.
\end{proof}

\begin{corollary}\label{indep}
   Suppose that $\varphi\colon E^n\to F$, where $E$ is a uniformly complete
   vector lattice,  $F$ is a vector lattice, $n\in\mathbb N$ and
  $\varphi$ is an orthosymmetric positive $n$-linear map. Then
  $\varphi(x_1,\dots,x_n)$ is determined by $(x_1\cdots
  x_n)^{\frac{1}{n}}$. Specifically,
  \begin{equation}
    \label{indep-fml}
    \varphi(x_1,\dots,x_n)=\varphi\bigl(x,\abs{x},\dots,\abs{x}\bigr)
  \end{equation}
  where $x=(x_1\cdots x_n)^{\frac{1}{n}}$.
\end{corollary}

\begin{proof}
  Suppose that $x_1,\dots,x_n\in E$. Let
  \begin{math}
    e=\abs{x_1}\vee\dots\vee\abs{x_n}
  \end{math}
  and consider the principal ideal $I_e$. Then $x_1,\dots,x_n\in
  I_e$. Since $I_e$ is lattice isomorphic to $C(K)$ for some compact
  Hausdorff space and the restriction of $\varphi$ to $(I_e)^n$ is
  still orthosymmetric, positive, and $n$-linear, by the theorem we
  get~\eqref{indep-fml}.
\end{proof}

\begin{remark}
  The expression
  \begin{math}
    \varphi\bigl(x,\abs{x},\dots,\abs{x}\bigr)
  \end{math}
  in~\eqref{indep-fml} may look non-symmetric at the first
  glance. Lemma~2 may be restated in a more ``symmetric'' form as
  follows:
  \begin{math}
    \varphi(x_1,\dots,x_n)=\varphi(x,\dots,x)
  \end{math}
  for every \emph{positive} $x_1,\dots,x_n$.
\end{remark}

Next, we are going to generalize Corollary~\ref{indep}.

\begin{theorem}\label{indep-p}
  Suppose that $\varphi\colon E_{(p_1)}\times\dots\times E_{(p_n)}\to
  F$, where $E$ is a uniformly complete vector lattice, $F$ is a
  vector lattice, $n\in\mathbb N$, $p_1,\dots,p_n$ are positive reals,
  and $\varphi$ is an orthosymmetric positive $n$-linear map. Then the
  following are true.
  \begin{enumerate}
  \item\label{p-indep} For all $x_1,\dots,x_n\in E$, we have
  \begin{math}
    \varphi(x_1,\dots,x_n)=\varphi\bigl(x,\abs{x},\dots,\abs{x}\bigr)
  \end{math}
  where $x=x_1^{p_1/p}\cdots x_n^{p_n/p}$ with $p=p_1+\dots+p_n$.
  \item\label{map-p-lin} The map
  $\hat\varphi\colon E_{(p)}\to F$ defined by
  \begin{math}
    \hat\varphi(x)=\varphi\bigl(x,\abs{x},\dots,\abs{x}\bigr)
  \end{math}
  is a positive linear map. If $\varphi$ is a lattice $n$-morphism
  then $\hat\varphi$ is a lattice homomorphism.
  \end{enumerate}
\end{theorem}

\begin{proof}
  \eqref{p-indep}
  First, we prove the statement for the case $E=C(K)$ for some
  Hausdorff compact space $K$. Define $\psi\colon E^n\to F$ via
  \begin{math}
    \psi(x_1,\dots,x_n)=
     \varphi\bigl(x_1^{1/p_1},\dots,x_n^{1/p_n}\bigr).
  \end{math}
  It is easy to see that $\psi$ is an orthosymmetric positive
  $n$-linear map. Hence, applying Theorem~\ref{CK-ones} to $\psi$, we
  get
  \begin{multline*}
    \varphi(x_1,\dots,x_n)
    =\psi\bigl(x_1^{p_1},\dots,x_n^{p_n}\bigr)
    =\psi\bigl(x_1^{p_1}\cdots x_n^{p_n},\one,\dots,\one\bigr)\\
    =\psi\bigl(x^p,\one,\dots,\one\bigr)
    =\psi\bigl(x^{p_1},\abs{x}^{p_2},\dots,\abs{x}^{p_n}\bigr)
    =\varphi\bigl(x,\abs{x},\dots,\abs{x}\bigr).
  \end{multline*}
  Now suppose that $E$ is a uniformly complete vector lattice. Choose
  $e\in E_+$ such that $x_1,\dots,x_n\in I_e$. Recall that $I_e$ is
  lattice isomorphic to $C(K)$ for some Hausdorff compact space
  $K$. It is easy to see that $(I_e)_{(p_i)}$ is an ideal in
  $E_{(p_i)}$. The restriction of $\varphi$ to
  \begin{math}
    (I_e)_{(p_1)}\times\dots\times(I_e)_{(p_n)}
  \end{math}
  is again an orthosymmetric positive $n$-linear map, so the
  conclusion follows from the first part of the proof.

  \eqref{map-p-lin} The proof that $\hat\varphi(\alpha\odot
  x)=\alpha\hat\varphi(x)$ is straightforward. We proceed to check
  additivity. Again, suppose first that $E=C(K)$ for some compact
  Hausdorff space $K$; let $\psi$ be as before. Take any $x,y\in E$
  and put $z=x\oplus y$ in $E_{(p)}$, i.e.,
  $z=(x^p+y^p)^{1/p}$. Then, again applying Theorem~\ref{CK-ones} to
  $\psi$, we have
  \begin{multline*}
   \varphi\bigl(z,\abs{z},\dots,\abs{z}\bigr)
   =\psi\bigl(z^{p_1},\abs{z}^{p_2},\dots,\abs{z}^{p_n}\bigr)
   =\psi(z^p,\one,\dots,\one)\\
   =\psi(x^p,\one,\dots,\one)+\psi(y^p,\one,\dots,\one)
   =\varphi\bigl(x,\abs{x},\dots,\abs{x}\bigr)
    +\varphi\bigl(y,\abs{y},\dots,\abs{y}\bigr).
  \end{multline*}
  Hence,
  \begin{equation}
  \label{sum}
    \varphi\bigl((x^p+y^p)^{\frac{1}{p}},\abs{x^p+y^p}^{\frac{1}{p}},\dots
      \abs{x^p+y^p}^{\frac{1}{p}}\bigr)=
    \varphi\bigl(x,\abs{x},\dots,\abs{x}\bigr)
     +\varphi\bigl(y,\abs{y},\dots,\abs{y}\bigr)
  \end{equation}
  Now suppose that $E$ is an arbitrary uniformly
  complete vector lattice and $x,y\in E$. Taking
  $e=\abs{x}\vee\abs{y}$ and proceeding as in~\eqref{p-indep}, one can
  see that~\eqref{sum} still holds,
  which yields $\hat\varphi(x\oplus y)=\hat\varphi(x)+\hat\varphi(y)$.
\end{proof}

\begin{corollary}\label{wgm}
  Let $E$ be a uniformly complete vector lattice and $p_1,\dots,p_n$
  positive reals; put $p=p_1+\dots+p_n$. For $x_1,\dots,x_n\in E$,
  define $\mu(x_1,\dots,x_n)=x_1^{p_1/p}\cdots
  x_n^{p_n/p}$. Then
  \begin{enumerate}
  \item\label{wgm-mu} $\mu\colon E_{(p_1)}\times\dots\times E_{(p_n)}\to
    E_{(p)}$ is an orthosymmetric lattice $n$-morphism;
  \item\label{wgm-corr} For every vector lattice $F$ there is a one to one
    correspondence between orthosymmetric positive $n$-linear maps
    $\varphi\colon E_{(p_1)}\times\dots\times E_{(p_n)}\to F$
    and positive linear maps $T\colon E_{(p)}\to F$ such that
    $\varphi=T\mu$ and
    $Tx=\varphi\bigl(x,\abs{x},\dots,\abs{x}\bigr)$. Moreover,
    $\varphi$ is a lattice $n$-morphism iff $T$ is a lattice
    homomorphism.
  \end{enumerate}
    \begin{figure}[!htb]\caption{}\label{mu-varphi}
    \begin{math}
      \xymatrix{
      E_{(p_1)}\times\dots\times E_{(p_n)} \ar[rr]^(.6)\varphi \ar[d]_\mu &
      & F \\
      E_{(p)} \ar[urr]^T &&
      }
     \end{math}
  \end{figure}
\end{corollary}

\begin{proof}
  \eqref{wgm-mu} is straightforward. Note that
  $\mu\bigl(x,\abs{x},\dots,\abs{x}\bigr)=x$ for every $x\in E$.

  \eqref{wgm-corr}
  If $T\colon E_{(p)}\to F$ is a positive linear map then setting
  $\varphi:=T\mu$ defines an orthosymmetric positive $n$-linear map on
  $E_{(p_1)}\times\dots\times E_{(p_n)}$ and
  \begin{displaymath}
    \varphi\bigl(x,\abs{x},\dots,\abs{x}\bigr)
    =T\mu\bigl(x,\abs{x},\dots,\abs{x}\bigr)
    =Tx.
  \end{displaymath}
  Conversely, suppose that  $\varphi\colon E_{(p_1)}\times\dots\times
  E_{(p_n)}\to F$ is an orthosymmetric positive $n$-linear map; define
  $T\colon E_{(p)}\to F$ via
  $Tx:=\varphi\bigl(x,\abs{x},\dots,\abs{x}\bigr)$. Then $T$ is a
  positive linear operator by Theorem~\ref{indep-p}\eqref{map-p-lin}.
  Given $x_1,\dots,x_n\in E$, put $x=\mu(x_1,\dots,x_n)$. It follows
  from Theorem~\ref{indep-p}\eqref{p-indep} that
  \begin{displaymath}
    T\mu(x_1,\dots,x_n)
    =Tx
    =\varphi\bigl(x,\abs{x},\dots,\abs{x}\bigr)
    =\varphi(x_1,\dots,x_n),
  \end{displaymath}
  so that $T\mu=\varphi$.
\end{proof}

We will use the fact, due to Luxemburg and Moore, that if $J$ is an
ideal in a vector lattice $F$ then the quotient vector lattice $F/J$
is Archimedean iff $J$ is uniformly closed, see, e.g., Theorem~2.23
in~\cite{Aliprantis:06} and the discussion preceding it. Recall that
given a set $A$ in a vector lattice $F$, $A$ is uniformly closed if
the limit of every uniformly convergent net in $A$ is contained in $A$
(it is easy to see that it suffices to consider sequences). The
uniform closure of a set $A$ in $F$ is the set of the uniform limits
of sequences in $A$; it can be easily verified that this set is uniformly
closed.
Clearly, the uniform closure of an ideal is an
ideal. Hence, for every set $A$, the uniform closure of the ideal
generated by $A$ is the smallest uniformly closed ideal containing
$A$.

For Archimedean vector lattices $E_1,\dots,E_n$, we write
$E_1\bar\otimes\dots\bar\otimes E_n$ for their Fremlin vector lattice
tensor product; see~\cite{Fremlin:72,Fremlin:74}.

\begin{theorem}
  Let $E$ be a uniformly complete vector lattice, $p_1,\dots,p_n$
  positive reals, and $\Io$ the uniformly closed ideal in
  $E_{(p_1)}\bar\otimes\dots\bar\otimes E_{(p_n)}$ generated by the
  elementary tensors of form $x_1\otimes\dots\otimes x_n$ with
  $\bigwedge_{i=1}^n\abs{x_i}=0$. Then the quotient
  \begin{math}
    \bigl(E_{(p_1)}\bar\otimes\dots\bar\otimes E_{(p_n)}\bigr)/\Io
  \end{math}
  is lattice isomorphic to $E_{(p)}$.
\end{theorem}

\begin{proof}
  Consider the diagram
  \begin{equation}
  \label{q-otimes}
     E_{(p_1)}\times\dots\times E_{(p_n)}
     \xrightarrow{\otimes}
     E_{(p_1)}\bar\otimes\dots\bar\otimes E_{(p_n)}
     \xrightarrow{q}
    \bigl(E_{(p_1)}\bar\otimes\dots\bar\otimes E_{(p_n)}\bigr)/\Io
  \end{equation}
  where $q$ is the quotient map; $q(u)=u+\Io=:\tilde u$ for $u\in
  E_{(p_1)}\bar\otimes\dots\bar\otimes E_{(p_n)}$.

  Let $\mu$ be as in Corollary~\ref{wgm}. By the universal property of the
  tensor product (see, e.g., \cite[Theorem~4.2(i)]{Fremlin:72}), there
  exists a lattice homomorphism
  \begin{math}
    M\colon E_{(p_1)}\bar\otimes\dots\bar\otimes E_{(p_n)}\to E_{(p)}
  \end{math}
  such that $M(x_1\otimes\dots\otimes x_n)=\mu(x_1,\dots,x_n)$ for all
  $x_1,\dots,x_n$. Since $\mu$ is orthosymmetric,
  $M(x_1\otimes\dots\otimes x_n)=0$ whenever
  $\bigwedge_{i=1}^n\abs{x_i}=0$. Since $M$ is a lattice homomorphism,
  it follows that $M$ vanishes on $\Io$. Therefore, the quotient
  operator $\widetilde M$ is well defined: for $u\in
  E_{(p_1)}\bar\otimes\dots\bar\otimes E_{(p_n)}$ we have $\widetilde
  M\tilde u=Mu$. Furthermore, since $q$ is a lattice homomorphism
  (see, e.g., \cite[Theorem~2.22]{Aliprantis:06}), it is easy to see
  that $\widetilde M$ is a lattice homomorphism as well.

  Note that $M$ is onto because for every $x\in E_{(p)}$ we have
  \begin{math}
    x=M\bigl(x\otimes\abs{x}\otimes\dots\otimes\abs{x}\bigr).
  \end{math}
  It follows that $\widetilde M$ is onto. It is left to show that
  $\widetilde M$ is one-to-one.

  The composition map $q\otimes$ in~\eqref{q-otimes} is an
  orthosymmetric lattice $n$-morphism. By Corollary~\ref{wgm},
  there is a lattice homomorphism
  \begin{math}
    T\colon E_{(p)}\to
    \bigl(E_{(p_1)}\bar\otimes\dots\bar\otimes E_{(p_n)}\bigr)/\Io
  \end{math}
  such that $q\otimes=T\mu$ and
  \begin{displaymath}
      Tx=(q\otimes)\bigl(x,\abs{x},\dots,\abs{x}\bigr)
       =x\otimes\abs{x}\otimes\dots\otimes\abs{x}+\Io
  \end{displaymath}
  for every $x\in E_{(p)}$.
  It follows that for every $x_1,\dots,x_n$ we have
  \begin{multline*}
    T\widetilde M(x_1\otimes\dots\otimes x_n+\Io)
   =TM(x_1\otimes\dots\otimes x_n)\\
   =T\mu(x_1,\dots,x_n)=(q\otimes)(x_1,\dots, x_n)
   =x_1\otimes\dots\otimes x_n+\Io,
  \end{multline*}
  so that $T\widetilde M$ is the identity on the quotient of algebraic tensor
  product
  $\bigl(E_{(p_1)}\otimes\dots\otimes E_{(p_n)}\bigr)/\Io$.
  We claim that it is still the identity map on
  \begin{math}
    \bigl(E_{(p_1)}\bar\otimes\dots\bar\otimes E_{(p_n)}\bigr)/\Io;
  \end{math}
  this would imply that $\widetilde M$ is one-to-one and complete the
  proof.

  Suppose that $u\in E_{(p_1)}\bar\otimes\dots\bar\otimes E_{(p_n)}$.
  By \cite[Theorem~4.2(i)]{Fremlin:72}, there exist
  $w:=z_1\otimes\dots\otimes z_n$ in $E_{(p_1)}\otimes\dots\otimes
  E_{(p_n)}$ with $z_1,\dots,z_n\ge 0$
  such that for every positive real $\delta$ there
  exists $v\in E_{(p_1)}\otimes\dots\otimes E_{(p_n)}$ with
  \begin{math}
    \abs{u-v}\le\delta w.
  \end{math}
  Since the quotient map $q$ is a lattice homomorphism, we get
  \begin{math}
    \abs{\tilde u-\tilde v}\le\delta \tilde w.
  \end{math}
  Since $T$ and $\widetilde M$ are lattice homomorphisms and, by the
  preceding paragraph,
  $T\widetilde M$ preserves $\tilde v$ and $\tilde z$, we get
  \begin{math}
    \abs{T\widetilde M\tilde u-\tilde v}\le\delta \tilde w.
  \end{math}
  It follows that
  \begin{math}
    \abs{T\widetilde M\tilde u-\tilde u}
    \le\abs{T\widetilde M\tilde u-\tilde v}+\abs{\tilde u-\tilde v}
    \le 2\delta\tilde w.
  \end{math}
  Since $\delta$ is arbitrary, it follows by the Archimedean property
  that $T\widetilde M\tilde u=\tilde u$.
\end{proof}

\begin{remark}
  Note that the lattice isomorphism constructed in
  the proof of the theorem sends
  \begin{math}
    x_1\otimes\dots\otimes x_n+\Io
  \end{math}
  into $x_1^{p_1/p}\cdots x_n^{p_n/p}$, while
  its inverse sends $x$ to $x\otimes\abs{x}\otimes\dots\otimes\abs{x}+\Io$
  for every $x$.
\end{remark}

\section{Products of Banach lattices}
\label{BL}

Now suppose that $E$ is a Banach lattice and $p$ is a positive real
number. For each $x\in E_{(p)}$ we define
\begin{displaymath}
  \norm{x}_{(p)}=
  \inf\Bigl\{\sum_{i=1}^k\norm{v_i}^p\mid
     \abs{x}\le v_1\oplus\dots\oplus v_k,\ v_1,\dots,v_k\ge 0\Bigr\}.
\end{displaymath}
It is easy to see that this is a lattice seminorm on $E_{(p)}$.  We
will write $x\sim y$ if the difference $x\ominus y$ is in the kernel
of this seminorm. For $x\in E$ we will write $[x]$ for the equivalence
class of $x$. Let $E_{[p]}$ be the completion of
$E_{(p)}/\ker\norm{\cdot}_{(p)}$. Then $E_{[p]}$ is a Banach lattice.

Let's compare this definition with the concepts of the
$p$-convexification and the $p$-concavification of a Banach lattice,
e.g., in~\cite{Lindenstrauss:79}. If $p>1$ and $E$ is $p$-convex then
$\norm{\cdot}_{(p)}$ is a complete norm on $E_{(p)}$, hence
$E_{[p]}=E_{(p)}$, and this is exactly the $p$-concavification of $E$
in the sense of~\cite{Lindenstrauss:79}. In particular, if $E$ is
$p$-convex with constant 1 then $\norm{\cdot}^p$ is already a norm, so
that, by the triangle inequality, we have
$\norm{\cdot}_{(p)}=\norm{\cdot}^p$. On the other hand, let
$0<p<1$. Put $q=\frac{1}{p}>1$. As in the construction of the
$q$-convexification $E^{(q)}$ of $E$ in~\cite{Lindenstrauss:79}, we
see that $\norm{\cdot}^p$ is already a norm on $E_{(p)}$, so that
$\norm{\cdot}_{(p)}=\norm{\cdot}^p$. In this case,
$E_{[p]}=E_{(p)}=E^{(q)}$. Thus, the $E_{[p]}$ notation allows us to
unify convexifications and concavifications, and it does not make any
assumptions on $E$ besides being a Banach (or even a normed) lattice.

If $E_1,\dots,E_n$ are Banach lattices, we write $E_1\ftp\dots\ftp
E_n$ for the Fremlin projective tensor of $E_1,\dots,E_n$ as
in~\cite{Fremlin:74}; we denote the norm on this product by
$\norm{\cdot}_\api$. We will make use of the following universal
property of this tensor product, which is essentially
Theorem~1E(iii,iv) in~\cite{Fremlin:74} (see also Part~(d) of
Section~2 in ~\cite{Schep:84}).

\begin{lemma}\label{univ-prop-BL}
  Suppose $E_1,\ldots,E_n$ and $F$ are Banach lattices. There is
  an one-to-one norm preserving correspondence between continuous
  positive $n$-linear maps $\varphi:E_1\times\ldots\times E_n\to F$
  and positive operators $\varphi^{\otimes}\colon E_1\ftp\dots\ftp
  E_n\to F$ such that
  $\varphi(x_1,\dots,x_n)=\varphi^{\otimes}(x_1\otimes\dots\otimes
  x_n)$.  Moreover, $\varphi^{\otimes}$ is a lattice homomorphism if
  and only if $\varphi$ is a lattice $n$-morphism.
\end{lemma}

\begin{lemma}\label{mu-bdd}
  Let $E$ be a Banach lattice and $\mu$ be as in Corollary~\ref{wgm}. Then
  $\norm{\mu}\le 1$.
\end{lemma}

\begin{proof}
  By Proposition~1.d.2(i) of~\cite{Lindenstrauss:79}, we have
  \begin{equation}\label{mu-norm-est}
    \bignorm{\mu(x_1,\dots,x_n)}
    \le\norm{x_1}^{\frac{p_1}{p}}\cdots\norm{x_n}^{\frac{p_n}{p}}
  \end{equation}
  for every $x_1,\dots,x_n$. Fix $x_1,\dots,x_n\in E$. As in the definition of
  $\norm{\cdot}_{(p)}$, suppose that
  \begin{equation}\label{vij}
    \abs{x_1}\le v^{(1)}_1\oplus\dots\oplus v^{(1)}_{k_1},
    \quad\dots,\quad
    \abs{x_n}\le v^{(n)}_1\oplus\dots\oplus v^{(n)}_{k_n}
  \end{equation}
  for some positive $v^{(m)}_{i}$'s. Since $\mu$ is a lattice
  $n$-morphism, we have
  \begin{displaymath}
    \bigabs{\mu(x_1,\dots,x_n)}
    =\mu\bigl(\abs{x_1},\dots,\abs{x_n}\bigr)
    \le\mu\Bigl(\bigoplus_{i_1=1}^{k_1}v^{(1)}_{i_1},\dots,
       \bigoplus_{i_n=1}^{k_n}v^{(n)}_{i_n}\Bigr)
    =\bigoplus_{i_1,\dots,i_n}\mu\bigl(v^{(1)}_{i_1},\dots,v^{(n)}_{i_n}\bigr)
  \end{displaymath}
  where each $i_m$ runs from $1$ to $k_m$. The
  definition of $\norm{\cdot}_{(p)}$ yields
  \begin{displaymath}
    \bignorm{\mu(x_1,\dots,x_n)}_{(p)}
    \le\sum_{i_1,\dots,i_n}\bignorm{\mu(v^{(1)}_{i_1},\dots,v^{(n)}_{i_n})}^p.
  \end{displaymath}
  It follows from~\eqref{mu-norm-est} that
 \begin{displaymath}
    \bignorm{\mu(x_1,\dots,x_n)}_{(p)}
    \le\sum_{i_1,\dots,i_n}\norm{v^{(1)}_{i_1}}^{p_1}\cdots\norm{v^{(n)}_{i_n}}^{p_n}
    =\Bigl(\sum_{i_1=1}^{k_1}\norm{v^{(1)}_{i_1}}^{p_1}\Bigr)\cdots
      \Bigl(\sum_{i_n=1}^{k_n}\norm{v^{(n)}_{i_n}}^{p_n}\Bigr).
  \end{displaymath}
  Taking infimum over all positive $v^{(m)}_{i}$'s in~\eqref{vij}, we
  get
  \begin{displaymath}
    \bignorm{\mu(x_1,\dots,x_n)}_{(p)}
    \le\norm{x_1}_{(p_1)}\cdots\norm{x_n}_{(p_n)}.
  \end{displaymath}
\end{proof}

\begin{theorem}\label{main}
  Let $E$ be a Banach lattice and $p_1,\dots,p_n$ positive reals. Put
  $F=E_{[p_1]}\ftp\dots\ftp E_{[p_n]}$. Let $\Ioc$ be the norm closed ideal in $F$
  generated by elementary tensors $[x_1]\otimes\dots\otimes[x_n]$ with
  $\bigwedge_{i=1}^n\abs{x_i}=0$. Then $F/\Ioc$ is lattice isometric
  to $E_{[p]}$ where $p=p_1+\dots+p_n$.
\end{theorem}

\begin{proof}
  Let $\mu$ be as in Corollary~\ref{wgm}. Fix $x_1,\dots,x_n$ in
  $E$. Take any $x_1',\dots,x_n'$ in $E$ such that
  $\norm{x_i'-x_i}_{(p_i)}=0$ as $i=1,\dots,n$. Then it follows from
  Lemma~\ref{mu-bdd} that
  \begin{displaymath}
    \bignorm{\mu(x_1,x_2,\dots,x_n)\ominus\mu(x_1',x_2,\dots,x_n)}_{(p)}
    =\bignorm{\mu(x_1\ominus x_1',x_2,\dots,x_n)}_{(p)}
    =0,
  \end{displaymath}
  so that $\mu(x_1,x_2,\dots,x_n)\sim\mu(x_1',x_2,\dots,x_n)$ in
  $E_{(p)}$.  Iterating this process, we see that
  \begin{math}
    \mu(x_1',\dots,x_n')\sim\mu(x_1,\dots,x_n).
  \end{math}
  It follows that $\mu$ induces a map
  \begin{displaymath}
    \tilde\mu\colon
      \bigl(E_{(p_1)}/\ker\norm{\cdot}_{(p_1)}\bigr)\times\dots\times
      \bigl(E_{(p_n)}/\ker\norm{\cdot}_{(p_n)}\bigr)\to
      E_{(p)}/\ker\norm{\cdot}_{(p)}
  \end{displaymath}
  via
  \begin{math}
    \tilde\mu\bigl([x_1],\dots,[x_n]\bigr)=\bigl[\mu(x_1,\dots,x_n)\bigr].
  \end{math}
  Lemma~\ref{mu-bdd} implies that $\norm{\tilde\mu}\le 1$, so
  that it extends by continuity to a map
  \begin{math}
    \varphi\colon E_{[p_1]}\times\dots\times E_{[p_n]}\to E_{[p]}.
  \end{math}
  It is easy to see that $\varphi$ is still an orthosymmetric lattice
  $n$-morphism and $\norm{\varphi}\le 1$. As in
  Lemma~\ref{univ-prop-BL}, $\varphi$ gives rise to a lattice
  homomorphism
  \begin{math}
    \varphi^\otimes\colon F\to E_{[p]}
  \end{math}
  such that $\norm{\varphi^\otimes}\le 1$ and
  \begin{equation}\label{phi-otimes}
    \varphi^\otimes\bigl([x_1]\otimes\dots\otimes[x_n]\bigr)=
    \varphi\bigl([x_1],\dots,[x_n]\bigr)=
    \bigl[\mu(x_1,\dots,x_n)\bigr].
  \end{equation}
  The latter implies that $\varphi^\otimes$ vanishes on $\Ioc$. This,
  in turn, implies that $\varphi^\otimes$ induces a map
  \begin{math}
    \widetilde{\varphi^\otimes}\colon F/\Ioc\to E_{[p]},
  \end{math}
  which is again a lattice homomorphism and
  $\bignorm{\widetilde{\varphi^\otimes}}\le 1$.

  Consider the map $\psi\colon E_{(p_1)}\times\dots\times
  E_{(p_n)}\to F/\Ioc$ defined by
  $\psi(x_1,\dots,x_n)=[x_1]\otimes\dots\otimes[x_n]+\Ioc$. It can be
  easily verified that $\psi$ is an orthosymmetric lattice
  $n$-morphism. It follows from Corollary~\ref{wgm} that there exists a
  lattice homomorphism $T\colon E_{(p)}\to F/\Ioc$ such that
  $\psi=T\mu$ and
  \begin{displaymath}
    Tx=\psi\bigl(x,\abs{x},\dots,\abs{x}\bigr)
      =[x]\otimes\bigl[\abs{x}\bigr]\otimes\dots\otimes
      \bigl[\abs{x}\bigr]+\Ioc
  \end{displaymath}
  for every $x\in E_{(p)}$.

  We claim that
  $\norm{Tx}\le\norm{x}_{(p)}$. Note first that as $\norm{\cdot}_\api$
  is a cross-norm, we have
  \begin{equation}\label{norm-Tx}
    \norm{Tx}
    \le\bignorm{[x]}_{E_{[p_1]}}\cdots\bignorm{[x]}_{E_{[p_n]}}
    \le\norm{x}_{(p_1)}\cdots\norm{x}_{(p_n)}
    \le\norm{x}^{p_1}\cdots\norm{x}^{p_n}
    =\norm{x}^p.
  \end{equation}
  Suppose that
  $\abs{x}\le v_1\oplus\dots\oplus v_m$ for some positive $v_1,\dots,v_m$, as in
  the definition of $\norm{\cdot}_{(p)}$. Then
  \begin{math}
    \abs{Tx}=T\abs{x}\le\sum_{i=1}^mTv_i,
  \end{math}
  so that
  \begin{math}
    \norm{Tx}\le\sum_{i=1}^m\norm{Tv_i}\le\sum_{i=1}^m\norm{v_i}^p
  \end{math}
  by~\eqref{norm-Tx}. It follows that
  $\norm{Tx}\le\norm{x}_{(p)}$.

  Therefore, $T$ induces an operator
  from $E_{(p)}/\ker\norm{\cdot}_{(p)}$ to $\Ioc$ and, furthermore, an
  operator from $E_{[p]}$ to $F/\Ioc$, which we will denote
  $\widetilde{T}$, such that $\widetilde{T}[x]=Tx$ for every $x\in
  E_{(p)}$. Clearly, $\widetilde{T}$ is still a lattice homomorphism
  and $\norm{\widetilde{T}}\le 1$.
  We will show that $\widetilde{T}$ is the inverse of
  $\widetilde{\varphi^\otimes}$. This will complete the proof because
  this would imply that $\widetilde{\varphi^\otimes}$ is a surjective lattice
  isomorphism; it would follow from
  $\norm{\widetilde{\varphi^\otimes}}\le 1$ and
  $\norm{\widetilde{T}}\le 1$ that $\widetilde{\varphi^\otimes}$ is an
  isometry.

  Take any $x\in E$ and consider the corresponding class $[x]$ in
  $E_{[p]}$. Using \eqref{phi-otimes}, we get
  \begin{displaymath}
    \widetilde{\varphi^\otimes}\widetilde{T}[x]
    =\widetilde{\varphi^\otimes}Tx
    =\varphi^\otimes\bigl([x]\otimes\bigl[\abs{x}\bigr]\otimes\dots\otimes\bigl[\abs{x}\bigr]\bigr)
    =\bigl[\mu\bigl(x,\abs{x},\dots,\abs{x}\bigr)\bigr]
    =[x].
  \end{displaymath}
  Therefore, $\widetilde{\varphi^\otimes}\widetilde{T}$ is the
  identity on $E_{[p]}$. Conversely, for any $x_1,\dots,x_n$ in $E$ it
  follows by~\eqref{phi-otimes} that
  \begin{multline*}
    \widetilde{T}\widetilde{\varphi^\otimes}
      \bigl([x_1]\otimes\dots\otimes[x_n]+\Ioc\bigr)
    =\widetilde{T}\bigl[\mu(x_1,\dots,x_n)\bigr]
    =T\mu(x_1,\dots,x_n)\\
    =\psi(x_1,\dots,x_n)
    =[x_1]\otimes\dots\otimes[x_n]+\Ioc.
  \end{multline*}
  Therefore, $\widetilde{T}\widetilde{\varphi^\otimes}$ is the
  identity on the linear subspace of $F/\Ioc$ that corresponds to the
  algebraic tensor product, i.e., on
  $q\bigl(E_{[p_1]}\otimes\dots\otimes E_{[p_n]}\bigr)$, where $q$ is
  the canonical quotient map from $F$ to $F/\Ioc$. Since
  $E_{[p_1]}\otimes\dots\otimes E_{[p_n]}$ is dense in $F$, it follows
  that $q$ maps it into a dense subspace of $F/\Ioc$. Therefore,
  $\widetilde{T}\widetilde{\varphi^\otimes}$ is the identity on a
  dense subspace of $F/\Ioc$, hence on all of $F/\Ioc$.
\end{proof}

\begin{remark}
  Note that the isometry from $F/\Ioc$ onto $E_{[p]}$ constructed in
  the proof of Theorem~\ref{main} sends
  \begin{math}
    [x_1]\otimes\dots\otimes[x_n]+\Ioc
  \end{math}
  into $\bigl[x_1^{p_1/p}\cdots x_n^{p_n/p}\bigr]$, while its
  inverse sends $[x]$ to
  \begin{math}
    [x]\otimes\bigl[\abs{x}\bigr]\otimes\dots\otimes
      \bigl[\abs{x}\bigr]+\Ioc
  \end{math}
  for every $x$.
\end{remark}

Applying the theorem with $p_1=\dots=p_n=1$, we obtain the following
corollary, which extends the main result of~\cite{BBPTT}; see
also~\cite{BB}.

\begin{corollary}
  Suppose that $E$ is a Banach lattice. Let $\Ioc$ be the closed ideal
  in $E\ftp\dots\ftp E$ generated by the elementary tensors
  $x_1\otimes\dots\otimes x_n$ where
  $\bigwedge_{i=1}^n\abs{x_i}=0$. Then $\bigl(E\ftp\dots\ftp E\bigr)/\Ioc$ is
  lattice isometric to $E_{[n]}$.
\end{corollary}

Recall that if $p<1$ then $E_{[p]}=E^{(q)}$, the $q$-convexification
of $E$ where $q=\frac{1}{p}$. Hence, putting $q_i=\frac{1}{p_i}$ in
the theorem, we obtain the following.

\begin{corollary}
  Suppose that $E$ is a Banach lattice $q_1,\dots,q_n$ are positive
  reals such that their geometric mean
  \begin{math}
    q:=\bigl(\frac{1}{q_1}+\dots+\frac{1}{q_n}\bigr)^{-1}
  \end{math}
  satisfies $q\ge 1$. Let $\Ioc$ be the closed ideal
  in $E^{(q_1)}\ftp\dots\ftp E^{(q_n)}$ generated by the elementary tensors
  $x_1\otimes\dots\otimes x_n$ where
  $\bigwedge_{i=1}^n\abs{x_i}=0$. Then
  $\bigl(E^{(q_1)}\ftp\dots\ftp E^{(q_n)}\bigr)/\Ioc$ is
  lattice isometric to $E^{(q)}$.
\end{corollary}


\end{document}